\documentclass{amsart}
\usepackage{amssymb, amsmath, amsthm}
\usepackage[all]{xy}
\usepackage{enumerate}

\newdir^{ (}{{}*!/-8pt/@^{(}}

\def\ZZ{\mathbb{Z}}
\def\PP{\mathbb{P}}
\def\P1{\PP^1}
\def\AAA{\mathbb{A}}
\def\A2{\AAA^2}

\def\GL{\mathrm{GL}}

\def\Ort{\mathrm{O}}
\def\SO{\mathrm{SO}}

\def\PGL{\mathrm{PGL}}

\def\ol{\overline}

\def\diag #1{\langle #1 \rangle}
\def\diagm #1{\mathrm{diag}(#1)}
\def\quat #1#2{\left(#1,#2\right)_2}

\newtheorem{theorem}{Theorem}[section]
\newtheorem{lemma}[theorem]{Lemma}

\theoremstyle{definition}

\newtheorem{remark}[theorem]{Remark}
\newtheorem{constr}[theorem]{Construction}
\newtheorem{example}[theorem]{Example}

\begin{document}

\title{Finite group actions on curves of genus zero}

\author{Mario Garcia-Armas}
\address{Department of Mathematics, University of British Columbia, Vancouver, \newline \indent %
BC V6T 1Z2, Canada} 
\email{marioga@math.ubc.ca}

\thanks{The author is partially supported by a Four Year Fellowship, University of British Columbia.}

\begin{abstract}
We classify, up to conjugacy, the finite (constant) subgroups $G$ of adjoint absolutely simple algebraic groups of type $A_1$ over an arbitrary field $k$ of characteristic not $2$.
\end{abstract}



\maketitle

\section{Introduction}

The finite subgroups of $\PGL_2(\mathbb{C})$ have been known for over a century: these are cyclic, dihedral and the so-called polyhedral groups $A_4$, $S_4$ and $A_5$ (see, e.g., \cite{Kl56}). Any two isomorphic finite subgroups of $\PGL_2(\mathbb{C})$ are conjugate. All of the above remains true if one replaces $\mathbb{C}$ by any algebraically closed field $k$ and asks about finite subgroups of order prime to $\mathrm{char}(k)$ (see \cite[\S 2.5]{Se72}). Using this result as a starting point, A. Beauville \cite{Beau10} classified, up to conjugacy, the finite subgroups $G$ of $\PGL_2(k)$ over an arbitrary field $k$, under the assumption that $|G|$ is prime to $\mathrm{char}(k)$. X. Faber \cite{Fa11} completed this picture by classifying the $p$-irregular subgroups of $\PGL_2(k)$, i.e., subgroups whose order is divisible by $p = \mathrm{char}(k)$, and describing their conjugacy classes.

In this paper, we are interested in studying the finite (constant) subgroups of (possibly non-split) adjoint absolutely simple algebraic groups of type $A_1$ over a field $k$, as well as their conjugacy classes. In doing so, we restrict our attention to fields $k$ of characteristic different from $2$.

It is well known that an adjoint absolutely simple algebraic group of type $A_1$ over $k$ is of the form $\PGL_1(A)$ for some quaternion algebra $A = \quat{a}{b}$, where $a, b \in k^{\times}$ (see, e.g., \cite[\S 26.A and 26.B]{KMRT98}). As usual, the symbol $\quat{a}{b}$ denotes the 4-dimensional $k$-algebra with basis $1, i, j, ij$, subject to the relations $i^2=a, j^2=b, ij=-ji$. Alternatively, we may regard $\PGL_1(A)$ as the automorphism group of the conic associated to the quadratic form $q = \diag{-a,-b,ab}$, which implies the existence of an isomorphism $\PGL_1(A) \cong \SO(q)$ (see, e.g., \cite[Cor. 69.6]{EKM08}). Recall that any smooth projective curve of genus $0$ is isomorphic to one such conic, so we completely classify faithful actions of finite groups on curves of genus $0$, up to equivalence.

We will thus be interested in finite constant subgroups of $\SO(q)$ for any ternary nondegenerate quadratic form $q$. Replacing $q$ by a scalar multiple of itself does not alter its isometry group, so we may assume throughout that $q$ has discriminant $1$.

We first deal with the $p$-irregular subgroups of $\SO(q)$. Interestingly, we show that this case reduces entirely to the classification in \cite{Fa11}.

\begin{theorem} \label{Theorem_p_irregular}
Let $k$ be a field of characteristic $p>2$ and suppose that $\SO(q)$ contains a $p$-irregular subgroup. Then the quadratic form $q$ is isotropic, i.e., there exists an isomorphism $\SO(q) \cong \PGL_2$.
\end{theorem}

It remains to classify the $p$-regular subgroups $G$ of $\SO(q)$, so we may assume henceforth that $\textrm{char}(k)$ is prime to $|G|$. Over an algebraic closure $\ol{k}$ of $k$, we have that $\SO(q)(\ol{k}) \cong \PGL_2(\ol{k})$. Thus any finite subgroup $G$ of $\SO(q)$ embeds into $\PGL_2(\ol{k})$, so it must be isomorphic to $\ZZ/n\ZZ$, $D_{2n}$ (the dihedral group of $2n$ elements), $A_4$, $S_4$ or $A_5$. Theorem \ref{Theorem_Groups} below classifies these subgroups up to isomorphism and Theorem \ref{Theorem_Conjugacy} up to conjugacy. Taking $q = \diag{-1,-1,1}$ in these theorems, we recover the results in \cite{Beau10}.

We will prove Theorem \ref{Theorem_Groups} in Section \ref{Sect_FiniteSubg}. Note that the classification of polyhedral groups in parts (b) and (c)
is hinted at in \cite{Se78}; here we make it explicit for completeness. Throughout the paper, we denote a primitive $n$-th root of $1$ by $\omega_n$, and we set $\alpha_n = (\omega_n + \omega_n^{-1})/2$ and $\beta_n = \alpha_n^2-1$.  

\begin{theorem} \label{Theorem_Groups}
Let $q$ be a nondegenerate quadratic form of discriminant $1$ and let $q_0 = \diag{1,1,1}$. 
\begin{enumerate}[\upshape(a)]
\item The group $D_{4} \cong (\ZZ/2\ZZ)^2$ is always contained in $\SO(q)$. For $n \geq 3$, the group $\SO(q)$ contains $\ZZ/n\ZZ$ and $D_{2n}$ if and only if $\alpha_n \in k$ and $q$ represents $-\beta_n$.
\item The group $\SO(q)$ contains $A_4$ and $S_4$ if and only if $q \simeq q_0$.
\item The group $\SO(q)$ contains $A_5$ if and only if $\sqrt{5}\in k$ and $q \simeq q_0$.
\end{enumerate}
\end{theorem}

We will prove Theorem \ref{Theorem_Conjugacy} in Section \ref{Sect_Conj}. Our argument relies on Galois cohomology techniques, building on the approach taken in [Beau10]. 

\begin{theorem} \label{Theorem_Conjugacy}
Let $q = \diag{-a,-b,ab}$ be a nondegenerate quadratic form.
\begin{enumerate}[\upshape(a)] 
\item The conjugacy classes of $\ZZ/2\ZZ$ inside $\SO(q)$ are in natural bijective correspondence with the set $D(q) \subset k^\times/k^{\times 2}$ consisting of nonzero square classes represented by $q$.
\item Let $Q_{a,b} = \{(x,y) \in (k^\times/k^{\times 2})^2|\quat{ax}{by} \cong \quat{a}{b}\}$. The symmetric group $S_3 = \{s,t|s^3 = t^2 = (st)^2 = 1\}$ acts on $Q_{a,b}$ by setting $s\cdot (x,y) = (-bxy,abx)$ and $t\cdot(x,y) = (x,-axy)$ for all $(x,y) \in Q_{a,b}$. Then the conjugacy classes of $(\ZZ/2\ZZ)^2$ inside $\SO(q)$ are in natural bijective correspondence with $Q_{a,b}/S_3$.
\item There is at most one conjugacy class of subgroups isomorphic to $\ZZ/n\ZZ$ ($n \geq 3$) inside $\SO(q)$.
\item Suppose that $D_{2n}$ is contained in $\SO(q)$ ($n \geq 3$). The set $D(\diag{1,-\beta_n})$ consisting of nonzero square classes represented by $\diag{1, -\beta_n}$ forms a subgroup of $k^\times/k^{\times 2}$. The class $\frac{\alpha_n+1}{2} \cdot k^{\times 2}$ is contained in $D(\diag{1,-\beta_n})$; let $C$ be the subgroup generated by this class.  Then the conjugacy classes of $D_{2n}$ inside $\SO(q)$ are in natural bijective correspondence with $D(\diag{1,-\beta_n})/C$. 
\item There is at most one conjugacy class of subgroups isomorphic to $A_4$, $S_4$ or $A_5$ inside $\SO(q)$.  
\end{enumerate}
\end{theorem} 

\smallskip
\noindent{\bf Acknowledgements.}
I thank Z. Reichstein for reading early versions of this paper and making several useful suggestions. I am also grateful to an anonymous referee for many helpful remarks.

\section{Finite subgroups of $\SO(q)$} \label{Sect_FiniteSubg}

\begin{lemma} \label{Lemma_Matrices_SO}
Let $q$ be a nondegenerate ternary quadratic form over $k$ and let $M$ be an element of $\SO(q)(k)$. Then the following results hold.
\begin{enumerate}[\upshape(a)]
\item Suppose that $M$ is diagonalizable over $\ol{k}$. Then its eigenvalues are $1$, $\lambda$, and $\lambda^{-1}$ for some $\lambda \in \ol{k}^{\times}$. If $\lambda \neq \pm 1$, then $q$ becomes isotropic over $k(\lambda)$, which is an extension of $k$ of order dividing $2$.
\item Suppose that $M$ is a nontrivial unipotent matrix. Then $q$ must be isotropic.
\end{enumerate}
\end{lemma}

\begin{proof}
Let $Q$ be the matrix associated to $q$. Note that $M^{-1} = Q^{-1} \,{}^T\!M\, Q$, whence the characteristic polynomial $P$ of $M$ satisfies $P(x) = -x^3 P(1/x)$. It follows easily that the eigenvalues of $M$ must be $1, \lambda, \lambda^{-1}$ for some $\lambda \in \ol{k}$. 

Suppose first that $M$ is diagonalizable. Taking the trace of $M$, we obtain that $\lambda+\lambda^{-1} \in k$, whence $[k(\lambda):k]$ is $1$ or $2$. We show that $q$ becomes isotropic over $k(\lambda)$, provided that $\lambda\neq \pm 1$. Indeed, over this field we can select an eigenvector $v$ of $M$ associated to the eigenvalue $\lambda$. Then, we compute $q(v) = q(Mv) = \lambda^2 q(v)$, whence $q(v)=0$. This completes the proof of part (a).

Suppose now that $M$ is a nontrivial unipotent matrix. Then we can find nonzero vectors $v_1, v_2$ such that $Mv_1 = v_1$ and $M v_2 = v_1 + v_2$ (this follows easily after conjugating $M$ into Jordan canonical form). Let $b_q$ be the symmetric bilinear form associated to $q$. Note that $b_q(v_1,v_2) = b_q(M v_1, M v_2) = q(v_1) + b_q(v_1,v_2)$, whence $q(v_1) = 0$. This finishes the proof.
%
\end{proof}

\begin{proof} [Proof of Theorem \ref{Theorem_p_irregular}]
Let $G$ be a $p$-irregular subgroup inside $\SO(q)$ and let $M \in G$ be any element of order $p$. Note that $0 = M^p - I = (M - I)^p$, whence $M$ is a unipotent matrix. By Lemma \ref{Lemma_Matrices_SO}(b), it follows that $q$ is isotropic, so we obtain that $\SO(q) \cong \PGL_2$.
\end{proof}

\begin{proof} [Proof of Theorem \ref{Theorem_Groups}]
(a) The first statement is trivial: the diagonal subgroup $D_0 \subset \SO(q)$ is isomorphic to $(\ZZ/2\ZZ)^2$. Assume henceforth that $n \geq 3$. It suffices to prove that $\ZZ/n\ZZ \subset \SO(q)$ $\Rightarrow$ $\alpha_n \in k$ and $q$ represents $-\beta_n$ $\Rightarrow$ $D_{2n} \subset \SO(q)$.

To prove the first implication, let $M$ be an element of order $n$ in $\SO(q)(k)$. By Lemma \ref{Lemma_Matrices_SO}(a), we may assume that the eigenvalues of $M$ are $1$, 
$\omega_n$ and $\omega_n^{-1}$, after replacing $M$ by a power of itself if necessary. Using Lemma \ref{Lemma_Matrices_SO}(a) again, we see that $\alpha_n \in k$ and $q$ becomes isotropic over $k(\omega_n) = k(\sqrt{\beta_n})$. It suffices to prove that $q$ is isotropic over $k(\sqrt{\beta_n})$ if and only if $q$ represents $-\beta_n$. If $q$ is isotropic over $k$, then $q$ is universal, so there is nothing to prove. Suppose that $q$ is anisotropic over $k$. It follows from \cite[Prop. 34.8]{EKM08} that $q \simeq q_0 \otimes \mathrm{N}_{k(\sqrt{\beta_n})/k} \perp q_1$ for some nondegenerate quadratic forms $q_0, q_1$, where $q_1$ is anisotropic over $k(\sqrt{\beta_n})$. If $q$ is isotropic over $k(\sqrt{\beta_n})$, we conclude that $q \neq q_1$ and thus $\dim(q_0)=\dim(q_1)=1$ (recall that $\mathrm{N}_{k(\sqrt{\beta_n})/k} \simeq \diag{1,-\beta_n}$). It follows that $q_1 \cong \diag{-\beta_n}$ by taking discriminants, whence $q$ represents $-\beta_n$. Conversely, suppose that the latter holds. Then we must have that $q \simeq \diag{-\beta_n,-\gamma,\beta_n\gamma}$ for some $\gamma \in k^\times$, so it follows that $q$ is isotropic over $k(\sqrt{\beta_n})$.

Suppose next that $\alpha_n \in k$ and $q$ represents $-\beta_n$. As we saw above, we may assume that $q = \diag{-\beta_n,-\gamma,\beta_n\gamma}$ for some $\gamma \in k^\times$. The matrices
$$
s = \left(\begin{array}{ccc}
1 & 0 & 0\\
0 & \alpha_n & \beta_n \\
0 & 1 & \alpha_n \\
\end{array}\right), \
t = \left(\begin{array}{ccc}
-1 & 0 & 0\\
0 & 1 & 0 \\
0 & 0 & -1 \\
\end{array}\right),
$$
are contained in $\SO(q)(k)$ and satisfy $s^n = t^2 = (st)^2 = 1$, whence they generate a subgroup isomorphic to $D_{2n}$. The proof of this part is complete.

(b) It clearly suffices to prove that $A_4 \subset \SO(q) \Rightarrow q \simeq q_0 \Rightarrow S_4 \subset \SO(q)$. Recall that $A_4$ acts linearly on $k^3$ by rotations on the tetrahedron with vertices $(\epsilon_1, \epsilon_2, \epsilon_1\epsilon_2)$, where $\epsilon_i = \pm 1$. This representation is absolutely irreducible and leaves $q_0$ invariant, i.e., we have a linear representation $\rho\colon A_4 \hookrightarrow \SO(q_0)(k)$. Recall that any absolutely irreducible representation of a finite group admits at most one invariant quadratic form, up to a scalar. Indeed, after passing to the algebraic closure, this is an immediate consequence of Schur's lemma. It follows that $q \simeq c \cdot q_0$ for some $c \in k^\times$; taking discriminants, we conclude that $c = 1$.

On the other hand, $S_4$ acts linearly on $k^3$ by rotations on the cube with vertices $(\pm 1, \pm 1, \pm 1)$ and the corresponding representation has trivial determinant. The form $q_0$ is invariant under this action and therefore $S_4$ embeds into $\SO(q_0)(k)$. If $q \simeq q_0$, then $\SO(q) \cong \SO(q_0)$ and the result follows.

(c) Note that $A_5 \subset \SO(q) \Rightarrow A_4 \subset \SO(q) \Rightarrow q \simeq q_0$ by part (b). Since $A_5$ contains elements of order $5$, it is necessary that $\omega_5 + \omega_5^{-1} \in k$, which happens if and only if $\sqrt{5} \in k$.

Conversely, if $\sqrt{5} \in k$, the group $A_5$ acts linearly on $k^3$ by rotations on the icosahedron with vertices $(\pm \varphi, \pm 1, 0), (0, \pm \varphi, \pm 1), (\pm 1, 0, \pm \varphi)$, where $\varphi = (1+\sqrt{5})/2$, and this action preserves $q_0$. The result readily follows.
\end{proof}

\begin{remark}
The conclusion in part (a) of the above theorem is independent of the choice of $\omega_n$. Indeed, an easy exercise on Chebyshev polynomials shows that $(\omega_n + \omega_n^{-1})/2 \in k$ and $q$ represents $1 - (\omega_n + \omega_n^{-1})^2/4 = -(\omega_n - \omega_n^{-1})^2/4$ if and only if the same holds for {\em every} $n$-th root of $1$. 
\end{remark}

\begin{example}
Let $k = \mathbb{Q}$. A primitive $n$-th root of $1$ is given by $\omega_n = \exp(2\pi i/n)$. Recall that $\omega_n + \omega_n^{-1} \in \mathbb{Q}$ if and only if $n = 1, 2, 3, 4$ or $6$. The group $\SO(q)$ contains $\ZZ/4\ZZ$ and $D_{8}$ if and only if $q$ represents $1$. Moreover, $\SO(q)$ contains $\ZZ/3\ZZ$ and $D_{6}$ if and only if it contains $\ZZ/6\ZZ$ and $D_{12}$ if and only if $q$ represents $3$. 
\end{example}

\section{Conjugacy classes of subgroups} \label{Sect_Conj}

The purpose of this section is to prove Theorem \ref{Theorem_Conjugacy}. We recall the following construction for convenience.

\begin{constr} (\cite[\S 2]{Beau10})
Let $G$ be an algebraic group defined over $k$ and let $H \subset G(k)$ be a subgroup. Fix a separable closure $k_s$ of $k$ and set $\Gamma = {\rm Gal}(k_s/k)$. Define the pointed set ${\rm Emb}_i(H,G(k))$ as the set of embeddings $H \hookrightarrow G(k)$ which are conjugate by an element of $G(k_s)$ to the natural inclusion $i\colon H \hookrightarrow G(k)$, modulo conjugacy by an element of $G(k)$. Also, define the pointed set ${\rm Conj}(H, G(k))$ consisting of subgroups of $G(k)$ which are conjugate to $H$ over $G(k_s)$, modulo conjugacy by an element of $G(k)$. 

The centralizer of $H$ in $G$, which we denote by $Z$, will be a closed subgroup of $G$ defined over $k$ (cf. \cite[Ch. 1, \S 1.7]{Bo91}). The kernel $H^1(k,Z)_0$ of the natural map $H^1(k,Z) \to H^1(k,G)$ is isomorphic to ${\rm Emb}_i(H,G(k))$ as pointed sets. The normalizer $N$ of $H$ in $G(k_s)$ acts on $1$-cocycles $\Gamma \to Z(k_s)$ in the following way: an element $n \in N$ sends $\sigma \mapsto a_\sigma$ to $\sigma \mapsto n^{-1} a_{\sigma} \sigma(n)$. This (right) action descends to $H^1(k,Z)$ and preserves $H^1(k,Z)_0$. Then there is an isomorphism of pointed sets between $H^1(k,Z)_0/N$ and ${\rm Conj}(H, G(k))$.
\end{constr}

Now recall that any two isomorphic finite subgroups (of order prime to $\textrm{char}(k)$) of $\SO(q)(k_s) \cong \PGL_2(k_s)$ are conjugate. Therefore, the conjugacy classes of finite subgroups of $\SO(q)$ of the same isomorphism type as some particular subgroup $H \subset \SO(q)$ are in natural bijective correspondence with ${\rm Conj}(H, \SO(q)(k))$, independently of the choice of $H$.

We now state some basic facts about the structure of $\SO(q)$. The proofs are easy and are left to the reader. In the sequel, we write $\diagm{a_1, \ldots, a_n}$ for the diagonal matrix with entries $a_1, \ldots, a_n$ along the diagonal.

\begin{lemma} \label{Lemma_Centralizer}
Let $q = \diag{-a,-b,ab}$ be a nondegenerate quadratic form. If $H$ is a finite subgroup of $\SO(q)$, let $Z$ be the centralizer of $H$ in $\SO(q)$ and let $N$ be the normalizer of $H$ in $\SO(q)(k_s)$.

\begin{enumerate}[\upshape(a)]
\item Let $H \cong \ZZ/2\ZZ$ be generated by the diagonal matrix $\mathrm{diag}(1,-1,-1)$. Then we have that
$$
Z = \left\{
\left(\begin{array}{cc}
\det M & 0\\
0 &  M\\
\end{array}\right):\ M \in \Ort(\diag{-b,ab})
\right\} \cong \Ort(\diag{-b,ab})
$$
and $N = Z(k_s)$.

\item Let $H \cong (\ZZ/2\ZZ)^2$ be the diagonal subgroup inside $\SO(q)$. Then we have that $Z=H$ and $N$ is isomorphic to $S_4$. Explicitly, if we set $u = \sqrt{-a}$ and $v = \sqrt{-b}$, the matrices
$$
\left(\begin{array}{ccc}
0 & vu^{-1} & 0\\
0 & 0 & u \\
v^{-1} & 0 & 0 \\
\end{array}\right), \
\left(\begin{array}{ccc}
-1 & 0 & 0\\
0 & 0 & -u \\
0 & -u^{-1} & 0 \\
\end{array}\right),
$$
generate a subgroup $N' \subset \SO(q)(k_s)$ isomorphic to $S_3$ and $N = H \rtimes N'$.
\item Let $n \geq 3$ and suppose that $H \cong \ZZ/n\ZZ$ is contained in $\SO(q)$. Using the same notation from Theorem \ref{Theorem_Groups}, we may assume that $q = \diag{-\beta_n,-\gamma,\beta_n \gamma}$ and $H$ is generated by the matrix
$$
s = \left(\begin{array}{ccc}
1 & 0 & 0\\
0 & \alpha_n & \beta_n \\
0 & 1 & \alpha_n \\
\end{array}\right).
$$
Then we have that
$$
Z = \left\{
\left(\begin{array}{cc}
1 & 0\\
0 & M\\
\end{array}\right):\ M \in \SO(\diag{-\gamma,\beta_n\gamma})
\right\} \cong \SO(\diag{-\gamma,\beta_n\gamma}).
$$
\item Let $n \geq 3$ and suppose that $H \cong D_{2n}$ is contained in $\SO(q)$. As before, assume that $q = \diag{-\beta_n,-\gamma,\beta_n \gamma}$ and $H$ is generated by the matrices
$$
s = \left(\begin{array}{ccc}
1 & 0 & 0\\
0 & \alpha_n & \beta_n \\
0 & 1 & \alpha_n \\
\end{array}\right), \
t = \left(\begin{array}{ccc}
-1 & 0 & 0\\
0 & 1 & 0 \\
0 & 0 & -1 \\
\end{array}\right).
$$
Then we have that $Z \cong \ZZ/2\ZZ$ is generated by $\diagm{1,-1,-1}$. Assume without loss of generality that $\omega_{2n} \in k_s$ satisfies $\omega_{2n}^2 = \omega_n$. Then the matrices 
$$
\ol{s} = \left(\begin{array}{ccc}
1 & 0 & 0\\
0 & \alpha_{2n} & 2\alpha_{2n}\beta_{2n} \\
0 & (2\alpha_{2n})^{-1} & \alpha_{2n} \\
\end{array}\right), \
t = \left(\begin{array}{ccc}
-1 & 0 & 0\\
0 & 1 & 0 \\
0 & 0 & -1 \\
\end{array}\right),
$$
generate $N \cong D_{4n}$ inside $\SO(q)(k_s)$.
\item Let $H = A_4,\ S_4$ or $A_5$ and suppose that $H$ is contained in $\SO(q)$. Then the centralizer $Z$ is trivial. 
\end{enumerate}
\qed
\end{lemma}

\begin{remark}
Since any two finite isomorphic subgroups $H_1$ and $H_2$ of $\SO(q)$ are conjugate over $k_s$, their centralizers will also be conjugate over $k_s$ (in particular, they must be isomorphic). However, they are not necessarily isomorphic over $k$. For a concrete example, take $k = \mathbb{R}$, $q=\diag{-1,-1,1}$, $H_1$ generated by $\diagm{1,-1,-1}$ and $H_2$ generated by $\diagm{-1,-1,1}$. A simple computation shows that $Z(H_1)\cong \Ort(\diag{-1,1})$ and $Z(H_2) \cong \Ort(\diag{-1,-1})$. These groups are not isomorphic; the identity component $Z(H_1)^\circ$ is isomorphic to $\mathbb{G}_m$ while $Z(H_2)^\circ \cong \SO_2$, which is a non-split torus over $\mathbb{R}$. 
\end{remark}
\begin{remark}
Suppose we are in the situation of Lemma \ref{Lemma_Centralizer}(c) with $\gamma = 1$. Then, $q = \diag{-\beta_n,-1,\beta_n} \simeq \diag{-1,-1,1}$ and $\SO(q) \cong \PGL_2$, so we are dealing with the case studied in \cite{Beau10}. Any two cyclic subgroups of order $n$ inside $\SO(q)$ are conjugate over $k$ (see Theorem \ref{Theorem_Conjugacy}), so the centralizer of such a subgroup is unique up to conjugacy. By Lemma \ref{Lemma_Centralizer}(c), it must be isomorphic to $\SO(\diag{-1,\beta_n})$, which is a split torus if and only if $\beta_n = \frac{1}{4}(\omega_n - \omega_n^{-1})^2 \in k^{\times 2}$ if and only if $\omega_n \in k$ (since $\omega_n+\omega_n^{-1} \in k$). So in general the centralizer is not isomorphic to the split torus $\mathbb{G}_m$, contrary to an assertion made in the proof of \cite[Thm. 4.2]{Beau10} and it might have nontrivial cohomology. However, the final result in \cite{Beau10} is unaffected since the map $H^1(k,Z) \to H^1(k,G)$ still has trivial kernel (see Theorem \ref{Theorem_Conjugacy}).
\end{remark}

We now recall some facts about the Galois cohomology of orthogonal groups of quadratic spaces (see \cite[\S 29.E]{KMRT98} for details). Let $q$ be any nondegenerate quadratic form of dimension $n$ defined over $k$. The cohomology set $H^1(k,\Ort(q))$ classifies isometry classes of $n$-dimensional nondegenerate quadratic forms over $k$, while $H^1(k,\SO(q))$ classifies isometry classes of $n$-dimensional quadratic forms $q'$ over $k$ such that $\mathrm{disc}(q') = \mathrm{disc}(q)$. The natural map $H^1(k,\SO(q)) \to H^1(k,\Ort(q))$ is injective. Let $D_0 \cong (\ZZ/2\ZZ)^{n-1}$ and $D \cong (\ZZ/2\ZZ)^n$ be the subgroups of diagonal matrices inside $\SO(q)$ and $\Ort(q)$, respectively. We have a commutative diagram
$$
\xymatrix{
1 \ar[r] & D_0 \ar[d]_-{i} \ar@{^{ (}->}[r] & D \ar[d]_-{j} \ar[r]^-{\det} & \ZZ/2\ZZ  \ar[r] & 1\\
&\SO(q) \ar@{^{ (}->}[r] & \Ort(q) 
}
$$
where the top row is exact. This induces a diagram on cohomology
$$
\xymatrix{
1 \ar[r] & H^1(k,D_0) \ar[d]_-{i_*} \ar[r] & (k^\times/k^{\times 2})^{n} \ar[d]_-{j_*} \ar[r]^{p} & k^\times/k^{\times 2} \ar[r] & 1 \\
&H^1(k,\SO(q)) \ar[r] & H^1(k,\Ort(q))
}
$$
where $p\colon (c_1, \ldots, c_n) \mapsto c_1 \ldots c_n$ is the product map. This identifies $H^1(k,D_0)$ with the elements $(c_1, \ldots, c_n) \in (k^\times/k^{\times 2})^{n}$ such that $c_1 \ldots c_n = 1\! \mod k^{\times 2}$. 

In what follows, we will abuse notation and refer to quadratic forms as elements of the cohomology sets $H^1(k,\SO(q))$ and $H^1(k,\Ort(q))$. The reader should bear in mind that we are tacitly referring to their isometry classes.
 
\begin{lemma} \label{Lemma_Cohom_Orthogonal}
Suppose that $q \simeq \diag{b_1, \ldots, b_n}$ is a nondegenerate quadratic form.  
\begin{enumerate}[\upshape(a)] 
\item The map $j_*$ takes $(c_1, \ldots, c_n)$ to $\diag{c_1 b_1, \ldots, c_n b_n}$ and consequently, $i_*$ takes $(c_1, \ldots, c_n)$, with $c_1 \ldots c_n = 1\! \mod k^{\times 2}$, to $\diag{c_1 b_1, \ldots, c_n b_n}$.
\item Let $q' \simeq \diag{d} \perp q$ and define $f\colon \SO(q) \to \SO(q')$ by sending
$$
M \mapsto 
\left(\begin{array}{cc}
1 & 0\\
0 & M\\
\end{array}\right).
$$
The induced map $f_*\colon H^1(k,\SO(q)) \to H^1(k,\SO(q'))$ sends
$$
q'' \mapsto \diag{d} \perp q''.
$$
\item Let $q' \simeq \diag{d} \perp q$ and define $f\colon \Ort(q) \to \SO(q')$ by sending
$$
M \mapsto 
\left(\begin{array}{cc}
\det M & 0\\
0 & M\\
\end{array}\right).
$$
The induced map $f_*\colon H^1(k,\Ort(q)) \to H^1(k,\SO(q'))$ sends
$$
q'' \mapsto \diag{\mathrm{disc}(q'') \mathrm{disc}(q')} \perp q''.
$$
\end{enumerate}
\end{lemma}

\begin{proof}
(a) This is certainly well known; for the lack of a direct reference, we supply a proof using Galois descent. Let $V$ be the $k$-vector space where $q$ is defined and write $q = \sum_i b_i v_i^* \otimes v_i^*$, where $v_1, \ldots, v_n$ is a basis of $V$.  

Let $c = (c_1, \ldots, c_n) \in (k^\times/k^{\times 2})^n$ and define $s_i \in k_s^\times$ satisfying $s_i^2=c_i \! \mod k^{\times 2}$. Fix a finite Galois extension $K/k$ containing the $s_i$ and set $\Gamma_K = \mathrm{Gal}(K/k)$. Recall that a $1$-cocycle $\Gamma_K \to (\ZZ/2\ZZ)^n$ representing $c$ is given by 
$$
\sigma \mapsto (s_1^{-1}\sigma(s_1), \ldots, s_n^{-1}\sigma(s_n)).
$$ 
Thus $j_*(c)$ is represented by the $1$-cocycle $a\colon\sigma \mapsto \diagm{s_1^{-1}\sigma(s_1), \ldots, s_n^{-1}\sigma(s_n)}$. The quadratic space associated to $a$ can be obtained by twisting the Galois action on the pair $(V_K,q_K)$ and taking the $\Gamma_K$-invariant elements. Note that the twisted action is defined on $v = \sum_i \lambda_i v_i \in V_K$ ($\lambda_i \in K$) as
$$
\sigma * v = \sum_i s_i^{-1} \sigma(s_i) \sigma(\lambda_i) v_i,\ \sigma \in \Gamma_K,
$$ 
and $v \in {}_aV_K$ is $\Gamma_K$-invariant if and only if $\lambda_i = s_i f_i$ for some $f_i \in k$ and all $i$. A $k$-basis of $W = ({}_aV_K)^{\Gamma_K}$ is given by $w_1 = s_1 v_1, \ldots, w_n = s_n v_n$. The corresponding quadratic form is 
$$
q' = \sum_i b_i v_i^* \otimes v_i^* = \sum_i s_i^2 b_i  (s_i^{-1}v_i^*) \otimes (s_i^{-1}v_i^*) \simeq \sum_i c_i b_i w_i^* \otimes w_i^*.
$$
This finishes the proof.

(b) Let $i_q\colon D_{0,q} \hookrightarrow \SO(q)$ and $i_{q'}\colon D_{0,q'} \hookrightarrow \SO(q')$ be the embeddings corresponding to the subgroups of diagonal matrices. It is easy to see that the restriction $f|_{D_{0,q}}\colon D_{0,q} \to D_{0,q'}$ induces a map $F\colon H^1(k,D_{0,q}) \to H^1(k,D_{0,q'})$ sending $(c_1, \ldots, c_n)$ to $(1, c_1, \ldots, c_n)$. Hence, if $q'' = \diag{x_1, \ldots, x_n}$ is any quadratic form such that $\mathrm{disc}(q'')=\mathrm{disc}(q)$, it follows that
$$
f_*(q'') = f_*(i_{q*}(x_1/b_1, \ldots, x_n/b_n)) = i_{q'*}(F(x_1/b_1, \ldots, x_n/b_n)) 
= \diag{d} \perp q''.
$$

(c) Let $j_q \colon D_{q} \hookrightarrow \Ort(q)$ and $i_{q'}\colon D_{0,q'} \hookrightarrow \SO(q')$ be as before. Note that the restriction $f|_{D_{q}}\colon D_{q} \to D_{0,q'}$ induces a map $F\colon H^1(k,D_{q}) \to H^1(k,D_{0,q'})$ sending $(c_1, \ldots, c_n)$ to $(c_1\ldots c_n, c_1, \ldots, c_n)$. The result follows using a similar argument to the one in part (b).
\end{proof}

We are ready to prove the main result of this section.

\begin{proof} [Proof of Theorem \ref{Theorem_Conjugacy}]
(a) Let $H \cong \ZZ/2\ZZ$ be generated by $\diagm{1,-1,-1}$ inside $\SO(q)$. By Lemma \ref{Lemma_Centralizer}(a), its centralizer $Z$ is isomorphic to $\Ort(\diag{-b,ab})$. By Lemma \ref{Lemma_Cohom_Orthogonal}(c), the natural inclusion $Z \hookrightarrow \SO(q)$ induces a map on cohomology $H^1(k,Z) \to H^1(k,\SO(q))$ sending a binary quadratic form $q'$ to $\diag{\mathrm{disc}(q')} \perp q'$. Hence, the kernel $H^1(k,Z)_0$ consists of the binary quadratic forms $q'$ such that $\diag{\mathrm{disc}(q')} \perp q' \simeq q$ (in particular, $\mathrm{disc}(q') \in D(q)$). Define a map $\Psi\colon H^1(k,Z)_0 \to D(q)$ sending $q' \mapsto \mathrm{disc}(q')$. If $q', q'' \in H^1(k,Z)_0$ satisfy $\Psi(q') = \Psi(q'')$, then $\diag{\Psi(q')} \perp q' \simeq \diag{\Psi(q'')} \perp q'' \simeq q$ implies $q' \simeq q''$ by Witt's Cancellation Theorem, so $\Psi$ is injective. To prove that $\Psi$ is surjective, let $d \in D(q)$ be arbitrary. Then $q = \diag{d} \perp q'$ for some quadratic form $q'$. Taking discriminants yields $\mathrm{disc}(q') = d$. This implies that $q' \in H^1(k,Z)_0$ and $\Psi(q') = d$. This proves that $H^1(k,Z)_0$ is in natural bijection with $D(q)$. Moreover, since the normalizer $N$ coincides with $Z(k_s)$, the action of $N$ on $H^1(k,Z)$ is trivial. This finishes the proof.

(b) Let $H \cong (\ZZ/2\ZZ)^2$ be the subgroup $D_0$ of diagonal matrices inside $\SO(q)$. By Lemma \ref{Lemma_Centralizer}(b), we have that $Z = H$ and the map $H^1(k,Z) \to H^1(k,\SO(q))$ induced by the natural inclusion sends $(x,y,z)$, with $xyz = 1\!\mod k^{\times 2}$, to $\diag{-ax,-by,abz}$. Therefore, $(x,y,z) \in H^1(k,Z)_0$ if and only if $\diag{-ax,-by,abz} \simeq \diag{-a,-b,ab}$, which is clearly equivalent to $(ax,by)_2 \cong (a,b)_2$. It follows easily that $H^1(k,Z)_0 \cong Q_{a,b}$.

We should now determine the action of the normalizer $N$ on $Q_{a,b}$. By Lemma \ref{Lemma_Centralizer}(b), we have $N = H \rtimes N'$, where $N' \cong S_3$. Since $H$ acts trivially on $H^1(k,Z)$, we only need to determine how $N'$ acts on $H^1(k,Z)_0$. Recall that $N'$ is generated by the matrices
$$
s = \left(\begin{array}{ccc}
0 & vu^{-1} & 0\\
0 & 0 & u \\
v^{-1} & 0 & 0 \\
\end{array}\right), \
t = \left(\begin{array}{ccc}
-1 & 0 & 0\\
0 & 0 & -u \\
0 & -u^{-1} & 0 \\
\end{array}\right),
$$   
where $u = \sqrt{-a}$ and $v = \sqrt{-b}$. Note that a $1$-cocycle $l\colon \mathrm{Gal}(k_s/k) \to Z(k_s)$ representing $(x,y) \in Q_{a,b}$ is given by 
$$
\sigma \mapsto l_{\sigma} = \diagm{x_1^{-1} \sigma(x_1), y_1^{-1} \sigma(y_1), x_1^{-1} \sigma(x_1) y_1^{-1} \sigma(y_1)},
$$
where $x_1^2 = x\! \mod k^{\times 2}$ and $y_1^2 = y\! \mod k^{\times 2}$. We compute
$$
s^{-1} l_\sigma \sigma(s) = 
\diagm{(v x_1 y_1)^{-1} \sigma(v x_1 y_1), (u v x_1)^{-1} \sigma(u v x_1), (u y_1)^{-1}\sigma(u y_1)},
$$
and
$$
t^{-1} l_\sigma \sigma(t) = 
\diagm{x_1^{-1} \sigma(x_1),(u x_1 y_1)^{-1} \sigma(u x_1 y_1),(u y_1)^{-1}\sigma(u y_1)}.
$$
Thus, the $1$-cocycles $\sigma \mapsto s^{-1} l_\sigma \sigma(s)$ and $\sigma \mapsto t^{-1} l_\sigma \sigma(t)$ correspond to the elements $(-bxy, abx)$ and $(x, -axy)$ in $Q_{a,b}$, respectively. Hence $N' \cong S_3$ acts on $Q_{a,b}$ as claimed and the result follows easily.

(c) We may assume that we are in the situation of Lemma \ref{Lemma_Centralizer}(c), i.e., $q = \diag{-\beta_n, -\gamma, \beta_n\gamma}$ and the centralizer $Z$ of $\ZZ/n\ZZ$ is isomorphic to $\SO(\diag{-\gamma,\beta_n\gamma})$. By Lemma \ref{Lemma_Cohom_Orthogonal}(b), the natural map $H^1(k,Z) \to H^1(k,\SO(q))$ sends a binary quadratic form $q'$ (with $\mathrm{disc}(q')=-\beta_n$) to $\diag{-\beta_n} \perp q'$. By Witt's Cancellation Theorem, the kernel $H^1(k,Z)_0$ is trivial and the claim follows.

(d) We may assume that $q = \diag{-\beta_n, -\gamma, \beta_n\gamma}$ and $H \cong D_{2n}$ is as in Lemma \ref{Lemma_Centralizer}(d). The centralizer $Z \cong \ZZ/2\ZZ$ is generated by $\diagm{1,-1,-1}$. Let $D_0 \subset \SO(q)$ be the subgroup of diagonal matrices; the natural inclusion $Z \hookrightarrow D_0$ induces a map $H^1(k,Z) \cong k^\times/k^{\times 2} \to H^1(k,D_0)$ sending $c \in k^\times/k^{\times 2}$ to $(1, c, c)$. Therefore the natural map $H^1(k,Z) \to H^1(k,\SO(q))$ sends $c \in k^\times/k^{\times 2}$ to $\diag{-\beta_n,-c\gamma,c\beta_n\gamma}$. By Witt's Cancellation Theorem, the kernel $H^1(k,Z)_0$ is given by those square classes $c$ such that $\diag{-c\gamma,c\beta_n\gamma} \simeq \diag{-\gamma,\beta_n\gamma}$. It follows easily that $c \in H^1(k,Z)_0$ if and only if $\diag{1,-\beta_n}$ represents $c$, i.e., $H^1(k,Z)_0 \cong D(\diag{1,-\beta_n})$.

We now consider the action of the normalizer $N \cong D_{4n}$ generated by 
$$
\ol{s} = \left(\begin{array}{ccc}
1 & 0 & 0\\
0 & \alpha_{2n} & 2\alpha_{2n} \beta_{2n} \\
0 & (2\alpha_{2n})^{-1} & \alpha_{2n} \\
\end{array}\right), \
t = \left(\begin{array}{ccc}
-1 & 0 & 0\\
0 & 1 & 0 \\
0 & 0 & -1 \\
\end{array}\right).
$$
Note that $H$ is the subgroup of $N$ generated by $s = \ol{s}^2$ and $t$. Clearly the action of $H$ on $H^1(k,Z)_0$ is trivial, so we only need to compute the action of $\ol{s}$ on $H^1(k,Z)_0$. Unraveling the identification $H^1(k,Z)_0 \cong D(\diag{1,-\beta_n})$, we see that a $1$-cocycle $l\colon \mathrm{Gal}(k_s/k) \to Z(k_s)$ representing $c \in H^1(k,Z)_0$ is given by $\sigma \to l_\sigma = \diagm{1,c_1^{-1}\sigma(c_1),c_1^{-1}\sigma(c_1)}$, where $c_1 \in k_s^\times$ satisfies $c_1^2 = c\! \mod k^{\times 2}$. Then we compute the $1$-cocycle
$$
\sigma \mapsto \ol{s}^{-1} l_\sigma \sigma(\ol{s}) = 
\diagm{1, (\alpha_{2n} c_1)^{-1}\sigma(\alpha_{2n} c_1), (\alpha_{2n} c_1)^{-1}\sigma(\alpha_{2n} c_1)}.
$$
It corresponds to the square class of $(\alpha_{2n} c_1)^2$, which is precisely $\frac{\alpha_n+1}{2} c$. This completes the proof.

(e) This is immediate from Lemma \ref{Lemma_Centralizer}(e).
\end{proof}

\begin{remark}
Suppose we are in the situation of Theorem \ref{Theorem_Conjugacy}(d) and $n$ is odd. Then a simple computation shows that $\frac{\alpha_n+1}{2} = \frac{1}{4}\left(\omega_n^{\frac{n+1}{2}} + \omega_n^{-\frac{n+1}{2}}\right)^2 \in k^{\times 2}$. Therefore, the conjugacy classes of $D_{2n}$ are in natural bijective correspondence with $D(\diag{1,-\beta_n})$ for $n$ odd. This is not necessarily true for even $n$.
\end{remark}

We now make the correspondences in parts (a), (b) and (d) of Theorem \ref{Theorem_Conjugacy} more explicit, by exhibiting representatives for each conjugacy class.
\begin{itemize}
\item Let $q = \diag{-a,-b,ab}$, let $d \in D(q)$ and let $q' = \diag{d,x,y}$ be a quadratic form isometric to $q$. Select $P \in \GL_3(k)$ such that $q = q' \circ P$. Then the element $d \in D(q)$ corresponds to the conjugacy class of the subgroup $P^{-1} H P \subset \SO(q)$, where $H \cong \ZZ/2\ZZ$ is generated by $\diagm{1,-1,-1}$.
\item Let $q = \diag{-a,-b,ab}$, let $(x,y) \in Q_{a,b}$ and let $q' = \diag{-ax,-by,abxy}$. Select $P \in \GL_3(k)$ such that $q = q' \circ P$. The element $(x,y)$ corresponds to the conjugacy class of the subgroup $P^{-1}D_0P \subset \SO(q)$, where $D_0  \cong (\ZZ/2\ZZ)^2$ is the subgroup of diagonal matrices in $\SO(q)$. Elements of $Q_{a,b}$ in the same $S_3$-orbit yield subgroups which are conjugate over $k$. 
\item Let $q = \diag{-\beta_n,-\gamma,\beta_n\gamma}$ and let $c \in D(\diag{1,-\beta_n})$. Then the quadratic form $q' = \diag{-\beta_n,-c\gamma,c\beta_n\gamma}$ is isometric to $q$, so we may select $P \in \GL_3(k)$ such that $q = q' \circ P$. The element $c$ corresponds to the conjugacy class of the subgroup $P^{-1} H P \subset \SO(q)$, where $H \cong D_{2n}$ is as in Lemma \ref{Lemma_Centralizer}(d). The elements $c$ and $\frac{\alpha_n+1}{2} c$ yield subgroups which are conjugate over $k$. 
\end{itemize}

\end{document}